\documentclass[12 pt]{amsart}

\usepackage{amsthm}
\usepackage{amsmath}
\usepackage{amssymb}
\usepackage{latexsym}
\usepackage{url}
\usepackage{graphicx}
\usepackage{enumerate}
\usepackage{fancyhdr}
\usepackage[margin=2.5cm]{geometry}

\DeclareMathOperator{\Tr}{Tr}

\author{Felix Goldberg}
\address{Caesarea-Rothschild Institute, University of Haifa, Haifa, Israel}
\email{felix.goldberg@gmail.com}

\author{Steve Kirkland}
\address{Department of Mathematics, University of Manitoba, Manitoba, Canada}
\email{Stephen.Kirkland@umanitoba.ca}

\author{Anu Varghese}
\address{Department of Mathematics, CUSAT, Cochin, India}
\email{anukarintholil@gmail.com}

\author{Ambat Vijayakumar}
\address{Department of Mathematics, CUSAT, Cochin, India}
\email{vambat@gmail.com}

\title{On split graphs with four distinct eigenvalues}
\date{April 30, 2014}

\newtheorem{thm}{Theorem}[section]

\newtheorem{rmrk}[thm]{Remark}
\newtheorem{lem}[thm]{Lemma}
\newtheorem{prop}[thm]{Proposition}
\newtheorem{prob}[thm]{Problem}

\newtheorem{defin}[thm]{Definition}

\newtheorem{qstn}[thm]{Question}

\begin{document}

\begin{abstract}
It is a well-known fact that a graph of diameter $d$ has at least  $d+1$ eigenvalues. Let us call a graph \emph{$d$-extremal} if it has diameter $d$ and exactly $d+1$ eigenvalues. Such graphs have been intensively studied by various authors.

A graph is \emph{split} if its vertex set can be partitioned into a clique and a stable set. Such a graph has diameter at most $3$. We obtain a complete classification of the connected bidegreed $3$-extremal split graphs. We also show how to construct certain families of non-bidegreed $3$-extremal split graphs. 
\end{abstract}

\subjclass[2010]{05C50, 05B05, 05C75}

\keywords{split graph, adjacency matrix, distinct eigenvalues, restricted eigenvalue, combinatorial design, graph structure}

\maketitle
 
\section{Introduction} 
The eigenvalues of a graph are the eigenvalues of its adjacency matrix $A(G)$. Let us denote the number of distinct eigenvalues of the graph $G$ by $\delta(G)$. 

It is a basic precept of spectral graph theory that low values of $\delta(G)$ indicate the presence of special structure in the graph $G$. Indeed, we may point out a number of classical results in this vein:

\begin{thm}\cite{Doob70}
If $\delta(G)=2$, then $G$ is isomorphic to a disjoint union of a number of copies of a clique. That is: $G=mK_{n}$
\end{thm}

\begin{thm}\cite{ShrBha65}
If $\delta(G)=3$ and $G$ is regular, then $G$ is strongly regular.
\end{thm}

\begin{thm}\cite[p. 166]{Spectra}\label{thm:regbip}
If $\delta(G)=4$ and $G$ is regular bipartite, then $G$ is the incidence graph of a symmetric design.
\end{thm}

For samples of some recent work of this kind we refer to \cite{FioGar06,Ste11}. A way to intuitively grasp why such results hold is to consider the minimal polynomial of $A(G)$ which factors linearly as $A(G)$ is diagonalizable. So if $\lambda_{1},\lambda_{2},\ldots,\lambda_{\delta}$ are the distinct eigenvalues of $G$, we have:
$$
(A-\lambda_{1}I)\cdot\ldots\cdot(A-\lambda_{\delta}I)=0.
$$
If $\delta$ is small, then we can infer from this equation constraints on the structure of $G$, using the fact the $ij$th entry of $A^{k}$ is the number of $k$-walks between vertices $i$ and $j$ (cf. \cite[p. 4]{Spectra_BH})). This is the approach taken by Doob in \cite{Doob70} and it works very well for $\delta=2$; for larger values of $\delta$ it becomes necessary to introduce additional assumptions on $G$ in order to complete the analysis.


Another relation between structure and spectrum is given by the following well-known fact:
\begin{prop}\cite[p. 5]{Spectra_BH}
Let $G$ be a connected graph of diameter $d$. Then $\delta(G)\geq d+1$.
\end{prop}

Let us call graphs of diameter $d$ who have $d+1$ distinct eigenvalues \emph{$d$-extremal}. The complete graph $K_{n}$ is $1$-extremal and strongly regular graphs are $2$-extremal. More generally, distance regular graphs are $d$-extremal \cite[p. 178]{Spectra_BH}. 

Our main goal in this paper is to obtain a result similar to Theorem \ref{thm:regbip} for the class of \emph{split} graphs, instead of bipartite graphs. Recall that a graph is split if its vertex set can be partitioned into a clique and a stable set. 

While split graphs are less celebrated than the bipartite ones, the reader will surely agree, upon reflection, that they form no less natural a class. Indeed, both notions of bipartite and split graphs have been jointly generalized: a graph has a \emph{$(k,\ell)$-partition} - or, shortly, the graph is $(k,\ell)$ - if its vertex set can be partitioned into $k$ independent sets and $\ell$ cliques (cf. \cite{Hell_et04}). Clearly, bipartite graphs are $(2,0)$ while split graphs are $(1,1)$. A further generalization can be found in \cite{Hell14}.

There are also other reasons to accord to split graphs an important role in graph theory. One of them is that the \emph{split partitions}, that is the degree sequence vectors of split graphs, form the top part of the lattice of graphic partitions \cite{Merris03}, which fact may ultimately account for the unexpected emergence of split graphs in various contexts. 

Another reason is the central role split graphs play in the class of chordal graphs. It has been shown in \cite{BenRicWor85} that almost all chordal graphs are split and understanding a property for split graphs is often a major stepping-stone on the way to undertsanding it for all chordal graphs.

Unlike for bipartite graphs, regularity does not seem to be a natural assumption for split graphs. We replace it instead with the assumption that all vertex degrees in $G$ are either $t$ or $y$ and say that $G$ is 
$(t,y)$-\emph{bidegreed} . The structure of such graphs is just as we would predict it to be:
\begin{lem}\cite[Theorem 2.1]{SplitIntegral}\label{lem:steve}
Let $G$ be a connected $(t,y)$-bidegreed split graph. Then all vertices in the clique are of degree $t$ and all vertices in the stable set are of degree $y$.
\end{lem}

Observe that split graphs have diameter of at most $3$. We can now pose our research problem as:
\begin{prob}
Characterize the $3$-extremal connected bidegreed split graphs.
\end{prob}

Our main result (Theorem \ref{thm:split4_class}) is that connected split bidegreed $3$-extremal graphs are either the coronas of cliques or are derived in a natural way from non-symmetric block designs with the property $r=\lambda^{2}$.

\section{Combinatorial preliminaries}

Let $D$ be a family of $k$-subsets of $E=\{x_{1},x_{2},\ldots,x_{v}\}$. The family $D$ is called a \emph{$(v,k,\lambda)$-design over $E$} if, for every two distinct elements $e,f$ of $E$, there are exactly $\lambda$ sets in $D$ that contain both $e$ and $f$. A design is called \emph{non-trivial} if $k<v$. We will assume $\lambda>0$ throughout, to avoid pathological cases.

The elements of $E$ are called the \emph{points} of $D$ while the sets in $D$ are called \emph{blocks} and their number is traditionally denoted by $b$. It is a well-known fact (cf. \cite[p. 4]{Stinson}) that every element of $E$ appears in the same number of blocks; this number is called the \emph{replication number} of $D$, traditionally denoted by $r$, and satisfies the equation $$r=\lambda\frac{v-1}{k-1}.$$
We shall sometimes find it convenient to expand the notation and speak of a $(v,b,r,k,\lambda)$-design.

Let us now define the split graph associated with the design $D$. Informally, we first start with the usual (bipartite) incidence graph $L(D)=(P \cup B,E)$ of $D$, so that $P$ and $D$ are the sets of points and blocks of $D$, respectively, and then add all possible edges between vertices in $P$, turning it into a clique. Formally, we can write:

\begin{defin}\label{def:assoc}
Let $D$ be a $(v,b,r,k,\lambda)$-design over $E$. The \emph{associated split graph} $G_{D}$ has $v+b$ vertices corresponding to the points and blocks of $D$. Two vertices $\epsilon,\nu$ in $G$ are adjacent if one of the following conditions holds: 
\begin{itemize}
\item
$\epsilon,\nu$ both correspond to points.
\item
$\epsilon$ corresponds to a point $x_{\epsilon} \in E$ and $\nu$ to a block $\mathbf{b}_{\nu} \in D$, so that $x_{\epsilon} \in \mathbf{b}_{\nu}$.
\end{itemize}
\end{defin}

It is easy to see that $G_{D}$ is indeed a split graph whose maximal clique $C$ has $c=v$ vertices and whose stable set has $s=b$ vertices. Observe that any $G_{D}$ must have diameter equal to $1$ or $2$ or $3$.

\begin{rmrk}
Notice that we do not rule out the possibility that $D$ has repeated blocks (so that the family $D$ is a multiset rather than a set).
\end{rmrk}


Next we present what is perhaps the earliest result of design theory (cf. \cite[p. 17]{Stinson}). 
\begin{thm}[Fisher's inequality]\label{thm:fisher}
Let $D$ be a non-trivial $(v,k,\lambda)$-design with $b$ blocks. Then $b \geq v$.
\end{thm}
If $b=v$ then the design is called \emph{symmetric} and if $b>v$ it is called \emph{non-symmetric}.
Another well-known fact that can be found on \cite[p. 23]{Stinson} is:
\begin{thm}\label{thm:sym}
Let $D$ be a $(v,k,\lambda)$-design. Then $D$ is symmetric if and only if any two blocks intersect in $\lambda$ points.
\end{thm}

\begin{lem}\label{lem:nonsym}
Let $D$ be a non-trivial $(v,k,\lambda)$-design with associated split graph $G_{D}$. Suppose that $G_{D}$ has diameter $3$, then $D$ is non-symmetric and $s>c$.
\end{lem}
\begin{proof}
Let $C$ be the clique and $S$ the independent set into which the vertex set of $G_{D}$ is partitioned. Suppose now for the sake of contradiction that $D$ is symmetric. Then by Theorem \ref{thm:sym} we see that every two blocks intersect. This means that every two vertices in $S$ have at least one common neighbour in $C$, implying that the diameter of $G$ is $2$ - a contradiction. Therefore $D$ must be non-symmetric and Fisher's inequality together with Theorem \ref{thm:sym} tells us that $s>c$.
\end{proof}

The \emph{incidence matrix} of a design $D$ is the $v \times b$ matrix $B=(b_{ij})$ with $b_{ij}=1$ if $x_{i}$ belongs to the $j$th block of $D$ and $b_{ij}=0$ otherwise.

\begin{lem}[{{\cite[Theorem 1.13]{Stinson}}}]\label{lem:design}
Let $B$ be a $v \times b$ matrix with values in $\{0,1\}$, such that each column of $B$ contains exactly $k$ $1$s. If $BB^{T}=\lambda J+(r-\lambda)I$ then $B$ is the incidence matrix of a $(v,b,r,k,\lambda)$-design $D$.
\end{lem}

Let $G$ be a graph. A partition $\pi$ of its vertex set $V(G)$ into cells $\pi_{1},\pi_{2},\ldots,\pi_{m}$ is called \emph{equitable} if any vertex $v \in \pi_{i}$ has $a^{\pi}_{ij}$ neighbours in $\pi_{j}$, irrespective of the choice of $v$. The partition can be described by a matrix: $A_{\pi}=(a^{\pi}_{ij})$. 

Equitable partitions have long been used in the study of adjacency matrices (cf. \cite[Section 9.3]{AGT2} or \cite[Section 2.4]{Eigenspaces}). Our Lemma \ref{lem:part} is a weaker version of  \cite[Theorem 2.4.6]{Eigenspaces}:
\begin{lem}\label{lem:part}
Let $G$ be a graph with equitable partition $\pi$. Then every eigenvalue of $A_{\pi}$ is an eigenvalue of $A$. 
Furthermore, the Perron values of $A$ and $A_{\pi}$ are equal.
\end{lem}

Finally, if $G$ is a graph, we denote by $G \circ K_{1}$ its \emph{corona} - the graph obtained by adding a pendant vertex to each vertex of $G$.


\section{Matrix-theoretic preliminaries}
The identity matrix will be denoted, as usual, $I$. The all-ones matrix, rectangular or square, according to context, will be denoted $J$. The all-ones vector will be denoted $j$. The set of eigenvalues of $A$ will be denoted $Spec(A)$.
The largest eigenvalue of a nonnegative matrix will be called its \emph{Perron value}.  The rank and trace of a matrix $A$ will be denoted $r(A)$ and $\Tr{A}$, respectively. The number of distinct eigenvalues of matrix $A$ will be denoted by $\delta(A)$. The Perron-Frobenius theorem (cf. \cite[Chapter 8]{HornJohnson}) will be used freely throughout.

We now record a number of simple matrix-theoretic lemmata.

\begin{lem}\label{lem:1}
If $M$ is a real symmetric matrix with $\delta(M)=1$ and eigenvalue $\lambda$, then $M=\lambda I$.
\end{lem}
\begin{proof}
Diagonalize $M$ as $F=VDV^{-1}$. Since $D=\lambda I$, $M=\lambda I$ follows immediately.
\end{proof}

If the row sums of a matrix all equal to the same number $\omega$ we will say that the matrix is $\omega$-stochastic. The next lemma is a standard fact.
\begin{lem}\label{lem:cj}
Let $M$ be a real symmetric $\omega$-stochastic matrix.
If $x$ is a $\gamma$-eigenvector of $M$ for some $\gamma \neq \omega$, then $j^{T}x=0$.
\end{lem}

\begin{lem}\label{lem:rank}
Let $M$ be a real symmetric $n \times n$ $\omega$-stochastic matrix with $\omega>0$ and let $\beta \in \mathbb{R}$ such that $\beta n\neq \omega$. Then $r(M)=r(M-\beta J)$.
\end{lem}
\begin{proof}
We can show that more is true, in fact: $\ker{M}=\ker{(M-\beta J)}$. Indeed, if $x \in \ker{M}$ then $j^{T}x=0$ by Lemma \ref{lem:cj} and therefore $Jx=0$ and $x \in \ker{(M-\beta J)}$. On the other hand, if $x \in \ker{(M-\beta J)}$ then $j^{T}x=0$ by Lemma \ref{lem:cj} and thus $Mx=\beta Jx=0$. 
\end{proof}

\begin{lem}\label{lem:cao}
Let $M$ be an irreducible nonnegative symmetric and $\omega$-stochastic $n \times n$ matrix with $\delta(M)=2$.
Let $\gamma \neq \omega$ be the other eigenvalue of $M$. Then $M=\frac{\omega-\gamma}{n}J+\gamma I$.
\end{lem}
\begin{proof}
According to \cite[p. 219-220]{Cao}, we have that $M=uu^{T}+\gamma I$, with $u$ being some positive vector. Since $Mj=\omega j$ we get $uu^{T}j+\gamma j=\omega j$. Set $f=u^{T}j$ and we can write $u=\frac{\omega-\gamma}{f}j$. Multiplying this equality by $j$ again we get $f=u^{T}j=n\frac{\omega-\gamma}{f}$ and thus $f^{2}=n(\omega-\gamma)$. Finally, $uu^{T}=\frac{(\omega-\gamma)^{2}}{f^{2}}J$ and we are done.
\end{proof}

\begin{defin}\cite[cf p. 117]{Spectra_BH}
Let $M$ be a real symmetric matrix. An eigenvalue of $M$ is called \emph{restricted} if it has an eigenvector that is orthogonal to $j$. The set of all restricted eigenvalues of $M$ will be denoted $R(M)$.
\end{defin}


We shall be interested in the $\omega$-stochastic case. The next lemma is a standard fact.
\begin{lem}\label{lem:rho}
Let $M$ be a real symmetric nonnegative $\omega$-stochastic matrix. Suppose $M$ is permuted to a block-diagonal form. Then the $\omega$-eigenspace consists of vectors constant on the indices corresponding to each diagonal block of $M$.
\end{lem}
\begin{lem}\label{lem:main}
Let $M$ be a real symmetric nonnegative $\omega$-stochastic matrix. 
Then $Spec(A) - \{\omega\} \subseteq R(M)$. Furthermore, $\omega \in R(M)$ if and only if $M$ is reducible.
\end{lem}
\begin{proof}
The first claim follows immediately from Lemma \ref{lem:cj}. 
If $M$ is irreducible, then $\omega$ is simple by the Perron-Frobenius theorem. Thus every $\omega$-eigenvector is a multiple of $j$ and so is $\omega$ is not restricted. 
Conversely, if $M$ is reducible, then we can find by Lemma \ref{lem:rho} a $\omega$-eigenvector that is orthogonal to $j$, and thus $\omega$ is restricted.
\end{proof}

Recall that the \emph{Schur complement} of the partitioned matrix  
$$
A=\left[
        \begin{array}{cc}
           X& Y\\
           Z& W
        \end{array}
    \right]
$$
is $A/_{X}=W-ZX^{-1}Y$, assuming that $X$ is invertible.
\begin{lem}\cite[Theorem 2.5]{Oue81}\label{lem:schur}
Suppose that $A$ has an invertible principal submatrix $X$. Then $r(A)=r(X)+r(A/_{X})$.
\end{lem}

\section{bidegreed split graphs with four eigenvalues}\label{sec:class}
Throughout this section we shall assume that $G$ is a connected split \emph{bidegreed} graph, that is that there are exactly two distinct vertex degrees in $G=(C,S)$. By Lemma \ref{lem:steve} we know that all vertices in $C$ share the same degree $d$ and all vertices in $S$ share the same degree $k$. A vertex in $C$ has $c-1$ neighbours inside $C$ and therefore $k^{'}=d-(c-1)$ neighbours in $C$. 
Double-counting the edges between $C$ and $S$ gives us:
$$
k^{'}c=sk \Rightarrow k^{'}=\frac{sk}{c}.
$$

Let us write down the adjacency matrix $A$ of the graph $G$, with the vertices of $C$ listed first and then those of $S$:
$$
A=\left[
        \begin{array}{cc}
           J-I& B\\
           B^{T}& 0
        \end{array}
    \right].
$$
The bidegreeness assumption means that the matrix $B$ satisfies $BJ=k^{'}J$ and $B^{T}J=kJ$. Therefore $BB^{T}J=kk^{'}J$ or, in other words, $BB^{T}$ is $kk^{'}$-stochastic.

Let us now consider an eigenvector corresponding to a nonzero eigenvalue $\mu$ of $A$:
\begin{equation*}\label{eq:x}
\left[
        \begin{array}{cc}
           J-I& B\\
           B^{T}& 0
        \end{array}
\right] 
\left[
        \begin{array}{c}
           x\\
           y
        \end{array}
\right]= \mu
\left[
        \begin{array}{c}
           x\\
           y
        \end{array}
\right].
\end{equation*}
Performing some obvious manipulations we deduce that $y=\mu^{-1}B^{T}x$ and that therefore 
\begin{equation}\label{eq:basic}
\mu Jx+BB^{T}x=(\mu+\mu^{2})x.
\end{equation}
This will be our basic equation. Now multiply both sides of \eqref{eq:basic} by $J$ on the left and obtain:
$$
\mu cJx+kk^{'}Jx=(\mu+\mu^{2})Jx. 
$$
Therefore we see that either $Jx=0$ or $\mu^{2}-(c-1)\mu-kk^{'}=0$ (or both). We are now in a position to describe the spectrum of $G$:
\begin{prop}\label{prop:desc}
Let $G=(C,S)$ be a connected bidegreed split graph. If $\mu$ is an eigenvalue of $A(G)$, then at least one of the following holds:
\begin{itemize}
\item
$\mu=0$.
\item
$\mu$ is a root of the quadratic equation $t^{2}-(c-1)t-kk^{'}=0$.
\item
For some nonzero $x$ with $j^{T}x=0$, we have $BB^{T}x=(\mu+\mu^{2})x$.
\end{itemize}
\end{prop}



The bidegreeness assumption tells us that $\pi=\{C,S\}$ is in fact an equitable partition of the vertices. The quotient matrix is:
$$
A_{\pi}=\left[
        \begin{array}{cc}
           c-1& k^{'}\\
           k& 0
        \end{array}
\right] 
$$

Therefore, from Lemma \ref{lem:part} we see that the roots of the equation $t^{2}-(c-1)t-kk^{'}=0$ are indeed always eigenvalues of $G$. Furthermore, the larger of these two roots is the Perron value. Let us call it $\rho$ and the other root $\psi$.



Let $\gamma$ be an eigenvalue of $BB^{T}$. Clearly $\gamma \geq 0$ because $BB^{T}$ is positive semidefinite. Let $P_{\gamma}$ denote the set of roots of the quadratic equation $t^{2}+t-\gamma=0$. Since the discriminant of the equation is $1+4\gamma>0$, we see that $|P_{\gamma}|=2$. Now let us reformulate Proposition \ref{prop:desc} in a more precise way:
\begin{prop}\label{prop:desc2}
Let $G=(C,S)$ be a connected bidegreed split graph. Then the spectrum of $A$ is:
$$
Spec(A)=\begin{cases}
\cup_{\gamma \in R(BB^{T})}{P_{\gamma}} \bigcup \{\rho,\psi\} & \text{, if } r(A)=c+s \\
\cup_{\gamma \in R(BB^{T})}{P_{\gamma}} \bigcup \{\rho,\psi,0\}   & \text{, if } r(A)<c+s. \\
\end{cases}
$$
\end{prop}

In fact, it turns out that there can be very little overlap between the different parts of the spectrum presented in Proposition \ref{prop:desc2}:
\begin{prop}\label{prop:d3}
\mbox{}\begin{enumerate}
\item
$\rho \notin P_{\gamma} , \forall \gamma \in R(BB^{T})$.
\item
$\psi \notin P_{kk^{'}}$.
\item
$0 \notin P_{\gamma}, \forall \gamma \neq 0$.
\item
$\gamma_{1} \neq \gamma_{2} \Rightarrow P_{\gamma_{1}} \cap P_{\gamma_{2}}=\emptyset$.
\end{enumerate}
\end{prop}
\begin{proof}
(1) Consider the Perron value $\rho$ of $A$. Since the graph is connected, $A$ is an irreducible matrix and there is positive eigenvector corresponding to $\rho$. But any eigenvector arising from a $\gamma \in R(BB^{T})$ will be orthogonal to $j$ by definition and so will have both positive and negative entries.
Therefore, $\rho$ cannot belong to any $P_{\gamma}$.

(2) First observe that $\psi=\frac{-kk^{'}}{\rho}$ by Vieta's formula. Now suppose that $\psi \in P_{kk^{'}}$ and let $\beta$ be the other member of $P_{kk^{'}}$. Then $\beta=\rho$ by Vieta's formula - a contradiction to statement (1) of this proposition, which we had proved already.

(3) Obvious.

(4) Suppose that $P_{\gamma_{1}}=\{\mu,\tau_{1}\},P_{\gamma_{2}}=\{\mu,\tau_{2}\}$. Then Vieta's formula tells us that $\mu+\tau_{1}=-1$ and $\mu+\tau_{2}=-1$, immediately implying $\tau_{1}=\tau_{2}$ and $\gamma_{1}=\gamma_{2}$.
\end{proof}

\begin{prop}\label{prop:count}
$\delta(A) \geq 2|R(BB^{T})|+1$.
\end{prop}
\begin{proof}
From part (1) of the previous proposition we know that $\rho$ is not contained in any $P_{\gamma}$. We also know that $|P_{\gamma}|=2$ for each $\gamma \in R(BB^{T})$.
\end{proof}


\begin{prop}\label{prop:rank}
$r(A) \leq 2c$.
\end{prop}
\begin{proof}
According to Lemma \ref{lem:schur}, $r(A)=r(J-I)+r(A/_{J-I})$. Since $(J-I)^{-1}=\frac{1}{c-1}J-I$, we have that $$A/_{J-I}=B^{T}(\frac{1}{c-1}J-I)B=-(B^{T}B-\frac{k^{2}}{c-1}J).$$
The rank of $J-I$ is $c$ while the rank of $B^{T}B-\frac{k^{2}}{c-1}J$ is equal, by Lemma \ref{lem:rank}, to $r(B^{T}B)$. Since $r(XX^{T})=r(X)$ we get that $r(A)=c+r(B)$
and finally, $r(B) \leq c$ since $B$ has $c$ rows. 
\end{proof}

Before we arrive at the culmination, let us observe that the diagonal entries of $BB^{T}$ are all equal to $k^{'}$ and therefore $$\Tr{BB^{T}}=ck^{'}.$$

We are now in a position to state and prove our main result:
\begin{thm}\label{thm:split4_class}
Let $G$ be a connected bidegreed split graph of diameter $3$, with maximal clique and stable set sizes $c,s$, respectively. 
Then $G$ has exactly four distinct eigenvalues if and only if it is of one of the following forms:
\begin{itemize}
\item
$G=K_{c} \circ K_{1}$.
\item
$G=G_{D}$ for a $(v=c,b=s,r,k,\lambda)$-design $D$ such that $r=\lambda^{2}$ and that $D$ has at least one pair of disjoint blocks.
\end{itemize}
\end{thm}\begin{proof}
Let us first prove that if $G$ has four distinct eigenvalues, then it must be of one of the forms we have indicated. Since $\delta(A)=4$ we immediately deduce from Proposition \ref{prop:count} that $BB^{T}$ can have at most one restricted eigenvalue. Let us first consider the case that $BB^{T}$ is reducible. Then from Lemma \ref{lem:main} we see that $kk^{'}$ is the only possible restricted eigenvalue of $BB^{T}$. Therefore $BB^{T}$ has exactly one distinct eigenvalue and $BB^{T}=kk^{'}I$ by Lemma \ref{lem:1}. Observe now that $\Tr{BB^{T}}=ck^{'}$ as remarked before; on this other hand, the trace must be equal to $c \cdot kk^{'}$. Therefore $k=1$.

Let us now pause to count the distinct eigenvalues of $A$, according to Proposition \ref{prop:desc2}: we have $\rho,\psi$ and two more in $P_{kk^{'}}$ and by Proposition \ref{prop:d3} these are four different numbers. Therefore we must have $r(A)=c+s$ (or otherwise $0$ would be an eigenvalue of $A$ as well, raising $\delta(A)$ to five and creating a contradiction). But from Proposition \ref{prop:rank} we have $r(A) \leq 2c$ and so $c \geq s$. This in turn implies $k^{'}=\frac{sk}{c}=\frac{s}{c} \leq 1$ and so $k=1$ and $c=s$. But this means that we have arrived at the conclusion that $G=K_{c} \circ K_{1}$.

Now we take up the case when $BB^{T}$ is irreducible. From Lemma \ref{lem:main} we know that $kk^{'}$ is not a restricted eigenvalue of $BB^{T}$ and so does not contribute to the spectrum of $A$. Therefore there is exactly one more restricted eigenvalue $\gamma$ which does contribute (if there were two, then we would have $\delta(A)\geq 5$, an impossibility). Since $kk^{'}$ is a simple eigenvalue of $BB^{T}$ we can easily determine $\gamma$:
$$
ck^{'}=\Tr{BB^{T}}=kk^{'}+(c-1)\gamma \Rightarrow \gamma=k^{'}\frac{c-k}{c-1}.
$$
We can apply Lemma \ref{lem:cao} to $BB^{T}$ (with $\omega=kk^{'}$, of course) and obtain that 
$$
BB^{T}=k^{'}\frac{k-1}{c-1}J+k^{'}\frac{c-k}{c-1}I.
$$

If we now let $\lambda=k^{'}\frac{k-1}{c-1}$ and $(r-\lambda)=k^{'}\frac{c-k}{c-1}$, then
Lemma \ref{lem:design} tells us that $B$ is the incidence matrix of some $(c,k,\lambda)$-design $D$ and therefore $G=G_{D}$. Furthermore, we know by Lemma \ref{lem:nonsym} that $s>c$ and therefore $r(A)<c+s$ and $0$ is an eigenvalue of $A$. 

Thus we see that to have $\delta(A)=4$ we must have $\psi \in P_{\gamma}$. This means that $\psi^{2}+\psi=\gamma$. To derive the implications of this condition we need to explicitly write out $\psi$ (something we have managed to avoid doing so far):
$$
\psi=\frac{(c-1)-\sqrt{(c-1)^{2}+4kk^{'}}}{2}.
$$
Therefore we see that:
$$
\psi+\psi^{2}=\frac{c(c-1)-c\sqrt{(c-1)^{2}+4kk^{'}}+2kk^{'}}{2}=\gamma=r-\lambda.
$$
Note that $r=\lambda+(r-\lambda)=k^{'}\frac{k-1}{c-1}+k^{'}\frac{c-k}{c-1}=k^{'}$ and therefore we can write $r$ instead of $k^{'}$. We are going to perform some algebraic manipulations, using the fact that $r(k-1)=\lambda(c-1)$:
$$
c\sqrt{(c-1)^{2}+4kr}=c(c-1)+2kr-2(r-\lambda)=c(c-1)+2(r(k-1)+\lambda)=
$$
$$
=c(c-1)+2(\lambda(c-1)+\lambda)=c(c-1)+2\lambda c.
$$
Dividing by $c$ we get:
$$
\sqrt{(c-1)^{2}+4kr}=(c-1)+2\lambda.
$$
Taking the square of both sides and simplifying then leads to
$$
rk=\lambda(c-1)+\lambda^{2}=r(k-1)+\lambda^{2} \Rightarrow r=\lambda^{2}.
$$

On the other hand, it is easy to verify by computation that graphs of the forms indicated in the theorem have exactly four distinct eigenvalues.


\end{proof}

\section{Some examples and discussion}\label{sec:fujiwara}

For examples of small combinatorial designs we shall draw on the tables in \cite[Chapter 1]{Designs} and on the online data on E. Spence's webpage \cite{SpenceDataNonSym}. The former provides exhaustive coverage up to $r=41$, with valuable commentary on the interrelationships between the various designs. The latter is less exhaustive (though it has some larger designs) but lists the actual incidence matrices for the designs, which can be very helpful in exploring their properties. 

The first parameter set that satisfies $r=\lambda^{2}$ is $(7,21,9,3,3)$. 
This is parameter set number 31 on \cite[p. 15]{Designs} and we can learn there that there are $10$ non-isomorphic such designs. We downloaded them from \cite{SpenceDataNonSym} and constructed the associated split graphs $G_{D}$ for each design. It turns out that one of the graphs (the first) has diameter $2$ while the rest have diameter $3$. The reason for the first graph having diameter $2$ is that every pair of blocks in the corresponding design has a non-empty intersection.

Indeed, let us list here the first design (with diameter $2$) and then the fifth and the tenth designs (with diameter $3$) in the compact notation of \cite{Designs}, as a $k \times b$ matrix, where each column lists the points of the corresponding block.

$$
\footnotesize 
D_{1}=\begin{pmatrix}
1&1&1&1&1&1&1&1&1&2&2&2&2&2&2&3&3&3&3&3&3\\
2&2&2&4&4&4&6&6&6&4&4&4&5&5&5&4&4&4&5&5&5\\
3&3&3&5&5&5&7&7&7&6&6&6&7&7&7&7&7&7&6&6&6\\
\end{pmatrix},
$$
$$
\footnotesize 
D_{5}=\begin{pmatrix}
1&1&1&1&1&1&1&1&1&2&2&2&2&2&2&3&3&3&3&3&3\\
2&2&2&4&4&4&5&5&6&4&4&4&5&5&6&4&4&4&5&5&6\\
3&3&3&5&6&7&6&7&7&5&6&7&6&7&7&5&6&7&6&7&7\\
\end{pmatrix},
$$
$$
\footnotesize 
D_{10}=\begin{pmatrix}
1&1&1&1&1&1&1&1&1&2&2&2&2&2&2&3&3&3&3&4&4\\
2&2&2&3&3&4&5&5&6&3&3&4&4&5&6&4&4&5&6&5&5\\
3&4&5&4&6&7&6&7&7&5&7&6&7&6&7&5&6&7&7&6&7\\
\end{pmatrix}.
$$

All in all there are ten non-isomorphic split bidegreed graphs with diameter $3$ and four distinct eigenvalues on $28$ vertices: the nine graphs associated with designs and $K_{14} \circ K_{1}$.

Observe that the block $\{1,2,3\}$ is repeated three times in $D_{5}$, whereas $D_{10}$ has no repeated blocks. Furthermore, observe that $D_{1}$ is obtained by replicating three times the blocks of the Fano plane (that is the unique $(7,3,1)$-design). We will now use this idea to produce infinite families of examples.

\subsection{The Fujiwara construction}
This construction reported in this subsection was suggested to us by Y. Fujiwara. The following fact is obvious from the definitions:

\begin{prop}\label{prop:f}
Let $D_{0}$ be a $(v,b,r,k,1)$-design and let $D$ be the family of sets consisting of $r$ copies of each block of $D_{0}$. Then $D$ is a $(v,br,r^{2},k,r)$-design.
\end{prop}

A design $D$ is called \emph{quasi-symmetric} if there exist parameters $x,y$ $(x<y)$ so that any two blocks of $D$ intersect in either $x$ or in $y$ elements. A quasi-symmetric design with $x=0$ is guaranteed to have a pair of non-intersecting blocks and thus to give rise to a $G_{D}$ with diameter $3$.
\begin{lem}\cite[p. 431]{Designs}\label{lem:xy}
Any $(v,k,1)$-design with $b>v$ must be quasi-symmetric with $x=0$ and $y=1$.
\end{lem}

A \emph{Steiner triple system} $STS(v)$ is a $(v,3,1)$-design. 
\begin{thm}\cite[p. 70]{Designs}
A Steiner triple system $STS(v)$ exists if and only if $v \equiv 1,3 \ (mod \ 6)$.
\end{thm}

Now let us take $D_{0}=STS(v)$. Then $D_{0}$ has replication number $\frac{v-1}{2}$. Let $D$ be the multiset consisting of $\frac{v-1}{2}$ copies of $D_{0}=STS(v)$ and note from Proposition \ref{prop:f} that $D$ is a $(v,3,\frac{v-1}{2})$-design which satisfies the condition $r=\lambda^{2}$. Finally, for $v \geq 8$ we will have $b=\frac{v(v-1)}{6}>v$ and thus $D_{0}$ (and $D$) will have a pair of disjoint blocks by Lemma \ref{lem:xy}.

\section{Some non-bidegreed split graphs with four eigenvalues}
The construction of split graphs from combinatorial designs with $r=\lambda^{2}$ can also produce examples of non-bidegreed split graphs with four eigenvalues - provided that we suitably generalize our notion of a design. 

\begin{defin}(cf. \cite[Chapter 38]{Designs})
Let $D$ be a family of subsets of $E=\{x_{1},x_{2},\ldots,x_{v}\}$. The family $D$ is called a \emph{$(r,\lambda)$-design over $E$} if for every two distinct elements $e,f$ of $E$ there are exactly $\lambda$ sets in $D$ that contain both $e$ and $f$. 
\end{defin}

A $(r,\lambda)$-design is in fact the structure obtained if we drop the condition that all blocks have the same size in a $(v,k,\lambda)$-design. If $B$ is the incidence matrix of a $(r,\lambda)$-design $D$ then we still have:
$$
BB^{T}=\lambda J+(r-\lambda)I.
$$
Reprising the calculations we have performed in Section \ref{sec:class} we see that the condition $r=\lambda^{2}$ ensures that $G_{D}$ indeed has exactly four distinct eigenvalues. Formally, we can state this as:
\begin{thm}\label{thm:rl}
Let $D$ be a $(r,\lambda)$-design with at least one pair of disjoint blocks. If $r=\lambda^{2}$, then $G_{D}$ is a split graph of diameter $3$ and with four distinct eigenvalues.
\end{thm}

\begin{figure}
\includegraphics[height=6cm,width=8cm]{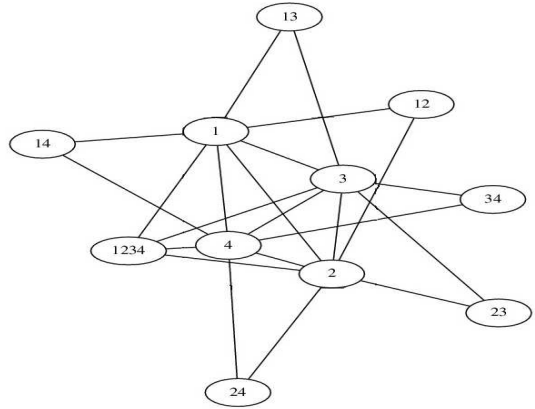}
\caption{A non-bidegreed split graph with $4$ distinct eigenvalues}\label{fig:graph}
\end{figure}

The graphs generated by Theorem \ref{thm:rl} can indeed be non-bidegreed. However, it is now incumbent upon us to give a few actual examples of such graphs. The first example is afforded by the $(4,2)$-design with the following blocks:
$$
1234, 12, 13, 14, 23, 24, 34.
$$
Taking $c=4$ and $s=7$ we obtain the graph of Figure \ref{fig:graph}. It has $11$ vertices and its spectrum is $5^{1},1^{3},0^{3},(-2)^{4}$ (the superscripts indicate multiplicities). The degree sequence is $7^{4},4^1,2^6$ and so the graph is indeed non-bidegreed. In Figure \ref{fig:graph}, the vertices of the clique are labelled $1,2,3,4$ and each vertex in the independent set is labelled with the list of points in the corresponding block.


Another class of examples can be obtained by the same approach as in Section \ref{sec:fujiwara}: if $D_{0}$ is any $(r,1)$-design, we can take $D$ to consist of $r$ copies of $D_{0}$ and then $D$ will be a $(r^{2},r)$-design. To illustrate this approach, we consider a $(6,1)$-design, denoted in \cite{LamReeVan91} as $CURD(9,6)$. Its blocks are:
$$
123\ 456\ 789\ 147\ 258\ 369
$$
$$
48\ 38\ 34\ 68\ 16\ 18
$$
$$
59\ 19\ 15\ 35\ 49\ 24
$$
$$
67\ 27\ 26\ 29\ 37\ 57
$$

Taking six copies of this design yields a split graph with $c=9$ and $s=144$. The degrees of the vertices in $S$ are all equal to either $2$ or $3$. We hope the reader will forgive us for not presenting here a drawing of the graph which has $153$ vertices. Its spectrum is $14^1,5^8,0^{135},-6^9$. 

We have not found in the literature a systematic treatment of the construction of $(r,\lambda)$-designs that satisfy $r=\lambda^{2}$. However we can at least refer the interested reader to the informal but informative discussion in \cite{MO120117} of a variety of ad-hoc methods by which such designs can be constructed from better-known designs. For constructions of $(r,1)$-designs we refer to \cite{Gro00}.


\section{Open questions}

\begin{qstn}
Let $G$ be a (not necessarily bidegreed) $3$-extremal split graph and let $t(G)$ be the number of distinct vertex degrees of $G$. Is there a constant $c \in \mathbb{N}$ such that $t(G) \leq c$? In particular, is it true that always $t(G) \leq 3$?
\end{qstn}

Similar questions are discussed in \cite[Section 4]{vanDam98}.

\begin{prob}
Characterize $d$-extremal regular chordal graphs.
\end{prob}

\section*{Acknowledgments}
We would like to thank Professor Mikhail Klin for valuable comments on the history of the subject and Professor Irene Sciriha for a careful reading of the paper. 

A part of this work was done while F.G. and S.K. were at the Hamilton Institute at the University of Maynooth, Ireland. The research of S.K. was supported in part by Science Foundation Ireland under grant number SFI/07/SK/I1216b and by the University of Manitoba under grant number 315729-352500-2000

\bibliographystyle{abbrv}
\bibliography{nuim}

\end{document}